\documentclass[preprint,1p,times,sort&compress]{elsarticle}
\usepackage{amsmath,amsfonts,mathrsfs}
\usepackage{amssymb}
\usepackage{color}
\usepackage{mathrsfs}
\usepackage{mathptmx}
\usepackage{verbatim}
\usepackage{epsfig}
\usepackage{CJK}
\usepackage{bm}
\usepackage{amsfonts}
\usepackage{amsthm}

\usepackage{amsthm}
\newtheorem{theorem}{Theorem}[section]
\newtheorem{lemma}{Lemma}[section]

\newdefinition{rem}{Remark}

\numberwithin{equation}{section}

\def\mi{{\bf i}}
\def\e{\varepsilon}
\journal{}

\begin{document}
\begin{frontmatter}
\title{A reducibility result for Schr\"{o}dinger operator with finite smooth and time quasi-periodic potential}

\author[SDUJ,SDUW]{Jing Li
}
 \ead{xlijing@sdu.edu.cn}
\address[SDUJ]{School of Mathematics , Shandong University, Jinan 250100, P.R. China}
\address[SDUW]{School of Mathematics and Statistics, Shandong University, Weihai 264209, P.R. China}

\begin{abstract}
In the present paper, we establish a reduction
theorem for linear Schr\"odinger equation with finite smooth and time-quasi-periodic potential
 subject to Dirichlet boundary condition by means of KAM technique.  Moreover, it is proved that 
 the corresponding Schr\"odinger operator possesses the property of pure point spectra and zero Lyapunov exponent.
\end{abstract}

\end{frontmatter}

%%%%%%%%%%%%%%%%%%%%%%%%%Main Text%%%%%%%%%%%%%%%%%%%%%%%%%%%%%%%%%%%%%%%%%%
\section{Introduction}

Recently years there are many literatures to investigate the reducibility for the linear Schr\"odinger equation of quasi-periodic potential, of the form
 \begin{equation}\label{1}\mi\, \dot u=(H_0+\e W(\omega t,x,-\mi \nabla))u,\quad x\in\mathbb R^d,\;\text{or}\; x\in\mathbb T^d=\mathbb R^d/(2\pi \mathbb Z)^d\end{equation}
or of the more general form, where $H_0=-\triangle +V(x)$ or an abstract self-adjoint (unbounded) operator and the perturbation $W$ is quasiperiodic in time $t$ and it may or may not depend on $x$ or/and $\nabla$.  From the reducibility it is proved immediately that  the corresponding  Schr\"odinger operator is of the pure point spectrum property and zero Lyapunov exponent.

 \ \

 When $x\in\mathbb R^d$, there are many interesting and important results. See \cite{Bambusi17arxiv,Bambusi17CMP,Combescure1987Ann,Wang17nonlinearity} and the references therein.

 \ \

 When $x\in\mathbb T^d$ with any integer $d\ge 1$, there are relatively less results. In \cite{eliasson-kuksin}, it is  proved that
\begin{equation}\label{2}\mi\, \dot u=\mi (-\triangle+\e W(\phi_0+\omega t,x;\omega)u),\;\; x\in\mathbb T^d \end{equation}
is reduced to an autonomous equation
for most values of the frequency vector $\omega$, where $W$ is analytic in $(t,x)$ and quasiperiodic in time $t$ with frequency vector $\omega$. The reduction is made by means
of T\"oplitz-Lipschitz property of operator and very hard KAM technique. As a special case of \eqref{2} with $d=1$, the reduction can be automatically derived from the earlier KAM theorem for nonlinear partial differential equations, assuming $W$ is analytic in $(t,x)$. See, \cite{Kuk1} and \cite{Poschel1996} for example.

\ \

As we know, the spectrum property depends heavily on the smoothness of the perturbation for the discrete Schr\"odinger operator. For example, the Anderson localization and positivity of the Lyapunov exponent for one frequency discrete quasi-period Schr\"odinger operator with analytic potential occur in non-perturbative sense (the largeness of the potential does not depend on the Diophantine condition. See \cite{Bourgain-Goldstin2000}, for the detail). However, one can only get perturbative results when the analytic property of  the potential is weaken to Gevrey regularity (see \cite{Klein2005}). Thus, a natural problem is that what happens when the perturbation $W$ is finite smooth in $(t,x)$.

\ \

Actually in his pioneer work, by reducibility Combescure \cite{Combescure1987Ann} studied the quantum stability problem for one-dimensional harmonic oscillator with a time-periodic perturbation. The
techniques of this paper were extended in \cite{Duclos-Stovicek96} and \cite{Duclos-Stovicek02}, in order to deal with  an abstract Schr\"odinger operator $ -\mi \partial_t + H_0 + \beta W(\omega t)$,
where $H_0$,  a selfadjoint operator acting in some Hilbert space , has simple discrete
spectrum $\lambda_n < \lambda_{n+1}$ obeying a gap condition of the type $\inf\{n^{-\alpha}(\lambda_{n+1}-\lambda_n):\; n= 1,2,...\} > 0$ for a given $\alpha>0,\beta\in\mathbb R$, and $W=W(t)$ is periodic in $t$ and $r$ times strongly
continuously differentiable as a bounded operator.

In the present paper, we will extend the time  periodic $W$ to time quasi-periodic one.  Let us consider a linear Schr\"{o}dinger equation with quasi-periodic coefficient:
\begin{equation}\label{eq1}%----------------------eq1.1
\mathcal{L} u\triangleq \mi \,u_{t}-u_{xx}+M u+\varepsilon W(\omega t,x)u=0
\end{equation}
subject to the boundary condition
\begin{equation}\label{eq2}%----------------------eq1.2
u(t,0)=u(t,\pi)=0,
\end{equation}
where $W$ is a quasi-periodic in time $t$ with frequency $\omega:$
$$ W(\omega t,x)=\mathcal{W}(\theta, x)|_{\theta=\omega t},
\;\;\mathcal{W}\in {C}^{N}(\mathbb{T}^{n}\times [0, \pi], \mathbb{R}),
\;\;\mathbb{T}^{n}=\mathbb{R}^{n}/(\pi \mathbb{Z})^{n},$$
and $W$ is also an even function of $x$.

Assume $\omega=\tau \omega_{0},$ where $\omega_{0}$ is Diophantine:

\begin{equation}\label{eq3}%----------------------eq1.3
|\langle k, \omega_{0} \rangle |\geq \frac{\gamma}{|k|^{n+1}},\;\;k\in \mathbb{Z}^{n}\setminus \{0\}
\end{equation}
with $0<\gamma\ll 1,$ a constant, and $\tau \in [1,2]$ is a parameter. Endow $L^2[0,\pi]\times L^2[0,\pi]$ a symplectic $\mi \, du \wedge d \bar u$. Take  $(L^2[0,\pi]\times L^2[0,\pi],\mi \, du \wedge d \bar u)$ as phase space. Then \eqref{eq1} is a hamiltonian system with hamiltonian functional
\[H(u,\bar u)=\int_{0}^{\pi}\frac12 |u_x|^2+\frac12(M u+\e W(\omega t,x))u\bar u\]

\begin{theorem}\label{thm1.1}
Assume that for a.e. $x\in [0, \pi],$ the potential $W(\cdot, x)$ is ${C}^{N}$
in the variable $\theta\in \mathbb{T}^{n}$ with $N\geq 80n.$ Then
there exists $ \;0<\varepsilon^{*}=\varepsilon^{*} (n, \gamma)\ll \gamma \ll 1,$
and exists a subset $\Pi\subset [1, 2]$ with
$$\mbox{Measure}\, \Pi\geq 1-O (\gamma)$$ and a quasiperiodic  symplectic change $u=\Psi(x,\omega t) v$
such that for any $\tau \in \Pi,$ \eqref{eq1} subject to \eqref{eq2} is changed into
\begin{equation} \mi \,v_{t}-v_{xx}+M_\xi v=0,\;\; v(t,0)=v(t,\pi)=0
\end{equation}
where $M_\xi$ is a real Fourier multiplier:
\[M_\xi \, \sin\, k x=(M+\xi_k) \sin\, k x,\; k\in \mathbb N,\]
with constants $\xi_k\in\mathbb R$ and $|\xi_k|\preceq \e$. Moreover, the Schr\"odinger operator $\mathcal L$ is of pure point spectrum property and of zero Lyapunov exponent.
\end{theorem}

\begin{rem} We will combine the Jackson-Moser-Zehnder approximation technique(see \cite{Chierchia}, for example) and KAM technique\cite{Kuk1} and \cite{Poschel1996}, which also applies to the case dealt with in \cite{Duclos-Stovicek96} and \cite{Duclos-Stovicek02}. Thus our result extends theirs. We also mention \cite{yuan-zhang} where the reducibility is dealt for a finite smooth and unbounded perturbation $W$.

\end{rem}

%----------------------------------------------------------------------------------------------------------
\section{Preliminaries}
\subsection{Analytical Approximation Lemma}

In this subsection, we cite an approximation lemma with the aim of this paper. These result can be obtained by \cite{Salamon1989} and \cite{Salamon2004}.

We start by recalling some definitions and setting some new notations. Assume $X$ is a Banach space with the norm
$||\cdot||_{X}$. First recall that $C^{\mu}(\mathbb{T}^{n}; X)$ for $0< \mu <1$ denotes the space of bounded
H\"{o}lder continuous functions $f: \mathbb{T}^{n}\mapsto X$ with the form
$$\|f\|_{C^{\mu}, X}=\sup_{0<|x-y|<1}\frac{\|f(x)-f(y)\|_{X}}{|x-y|^{\mu}}+\sup_{x\in \mathbb{T}^{n}}\|f(x)\|_{X}.$$
If $\mu=0$ then $\|f\|_{C^{\mu},X}$ denotes the sup-norm. For $\ell=k+\mu$ with $k\in \mathbb{N}$ and $0\leq \mu <1,$
we denote by $C^{\ell}(\mathbb{T}^{n};X)$ the space of functions $f:\mathbb{T}^{n}\mapsto X$ with H\"{o}lder continuous partial derivatives, i.e., $\partial ^{\alpha}f\in C^{\mu}(\mathbb{T}^{n}; X_{\alpha})$ for all muti-indices $\alpha=(\alpha_{1}, \cdots, \alpha_{n})\in \mathbb{N}^{n}$ with the assumption that
$|\alpha|:=|\alpha_{1}|+\cdots+|\alpha_{n}|\leq k$ and $X_{\alpha}$ is the Banach space of bounded operators
$T:\prod^{|\alpha|}(\mathbb{T}^{n})\mapsto X$ with the norm
$$\|T\|_{X_{\alpha}}=\sup \{||T(u_{1}, u_{2}, \cdots, u_{|\alpha|})||_{X}:\|u_{i}\|=1, \;1\leq i \leq |\alpha|\}.$$
We define the norm
$$||f||_{C^{\ell}}=\sup_{|\alpha|\leq \ell}||\partial ^{\alpha}f||_{C^{\mu}, X_{\alpha}}$$
\begin{lemma}(Jackson, Moser, Zehnder)\label{Jackson}
Let $f\in C^{\ell}(\mathbb{T}^{n}; X)$ for some $\ell>0$ with finite $C^{\ell}$ norm over $\mathbb{T}^{n}.$
Let $\phi$ be a radical-symmetric, $C^{\infty}$ function, having as support the closure of the unit ball centered at the origin, where $\phi$ is completely flat and takes value 1, let $K=\widehat{\phi}$ be its Fourier transform. For all $\sigma >0$ define
$$ f_{\sigma}(x):=K_{\sigma}\ast f=\frac{1}{\sigma^{n}}\int_{\mathbb{T}^{n}}K(\frac{x-y}{\sigma})f(y)dy.$$
Then there exists a constant $C\geq 1$ depending only on $\ell$ and $n$ such that the following holds: For any $\sigma >0,$ the function $f_{\sigma}(x)$ is a real-analytic function from $\mathbb{C}^{n}/(\pi \mathbb{Z})^{n}$ to $X$ such that if $\Delta_{\sigma}^{n}$ denotes the $n$-dimensional complex strip of width $\sigma,$
$$\Delta_{\sigma}^{n}:=\{x\in \mathbb{C}^{n}/(\pi \mathbb{Z})^{n}\big ||\mathrm{Im} x_{j}|\leq \sigma,\;1\leq j\leq n\},$$
then for $\forall \alpha\in\mathbb{N}^{n}$ such that $|\alpha|\leq \ell$ one has

\begin{equation}\label{cite3.1}
\sup_{x\in \Delta_{\sigma}^{n}}||\partial ^{\alpha}f_{\sigma}(x)
-\sum_{|\beta|\leq \ell-|\alpha|}\frac{\partial^{\beta+\alpha}f(\mathrm{Re}x)}{\beta !}(\sqrt{-1}\mathrm{Im}x)^{\beta}||_{X_{\alpha}}\leq C ||f||_{C^{\ell}}\sigma^{\ell-|\alpha|},
\end{equation}

and for all $0\leq s\leq \sigma,$
\begin{equation}\label{cite3.2}
\sup_{x\in \Delta_{s}^{n}}\|\partial^{\alpha} f_{\sigma}(x)-\partial^{\alpha}f_{s}(x)\|_{X_{\alpha}}
\leq C ||f||_{C^{\ell}}\sigma^{\ell-|\alpha|}.
\end{equation}

The function $f_{\sigma}$ preserves periodicity (i.e., if $f$ is T-periodic in any of its variable $x_{j}$, so is $f_{\sigma}$). Finally, if $f$ depends on some parameter $\xi\in \Pi\subset\mathbb{R}^{n}$ and
if the Lipschitz-norm of $f$ and its $x$-derivatives are uniformly bounded by $\|f\|_{C^{\ell}}^{\mathcal{L}},$
then all the above estimates hold with $\|\cdot\|$ replaced by $\|\cdot\|^{\mathcal{L}}.$
\end{lemma}

For the following result, the reader can referred to \cite{yuan-zhang}, for detail.
For ease of notation, we shall replace $\|\cdot\|_{X}$ by
$\|\cdot\|.$
Fix a sequence of fast decreasing numbers $s_{\nu}\downarrow 0, \upsilon\geq 0,$ and $s_{0}\leq \frac{1}{2}.$
For a $X$-valued function $P(\phi),$
construct a sequence of real analytic functions $P^{(\upsilon)}(\phi)$ such that the following conclusions holds:
\begin{description}
  \item[(1)] $P^{(\upsilon)}(\phi)$ is real analytic on the complex strip $\mathbb{T}^{n}_{s_{\upsilon}}$ of the width $s_{\upsilon}$ around $\mathbb{T}^{n}.$
  \item[(2)] The sequence of functions $P^{(\upsilon)}(\phi)$ satisfies the bounds:
  \begin{equation}\label{cite3.3}
  \sup_{\phi\in\mathbb{T}^{n}}\|P^{(\upsilon)}(\phi)-P(\phi)\|\leq C \|P\|_{C^{\ell}}s_{\upsilon}^{\ell},
  \end{equation}
  \begin{equation}\label{cite3.4}
  \sup_{\phi\in\mathbb{T}^{n}_{s_{\upsilon+1}}}\|P^{(\upsilon+1)}(\phi)-P^{(\upsilon)}(\phi)\|\leq C \|P\|_{C^{\ell}}s_{\upsilon}^{\ell},
  \end{equation}
  where $C$ denotes (different) constants depending only on $n$ and $\ell.$
  \item[(3)] The first approximate $P^{(0)}$ is "small" with the perturbation $P$. Precisely speaking, for arbitrary $\phi\in\mathbb{T}^{n}_{s_{0}},$ we have
      \begin{eqnarray}\label{cite3.5}
      \|P^{(0)}(\phi)\|&\leq &\|P^{(0)}(\phi)-\sum_{|\alpha|\leq \ell}\frac{\partial^{\alpha}P(\mathrm{Re \phi})}{\alpha!}(\sqrt{-1}\mathrm{Im}\phi)^{\alpha}\|+\|\sum_{|\alpha|\leq \ell}\frac{\partial^{\alpha}P(\mathrm{Re \phi})}{\alpha!}(\sqrt{-1}\mathrm{Im}\phi)^{\alpha}\|\nonumber\\
      &\leq & C(\|P\|_{C^{\ell}}s_{0}^{\ell}+\sum_{0\leq m\leq\ell}\|P\|_{C^{m}}s_{0}^{m})\nonumber\\
      &\leq & C\|P\|_{C^{\ell}}\sum_{m=0}^{\ell}s_{0}^{m}\nonumber\\
      &\leq & C\|P\|_{C^{\ell}}\sum_{m=0}^{\infty}s_{0}^{m}\nonumber\\
      &\leq & C\|P\|_{C^{\ell}},
      \end{eqnarray}
      where constant $C$ is independent of $s_{0},$ and the last inequality holds true due to the hypothesis that
      $s_{0}\leq \frac{1}{2}.$
  \item[(4)] From the first inequality \eqref{cite3.3}, we have the equality below. For arbitrary $\phi\in\mathbb{T}^{n},$
      \begin{equation}\label{cite3.6}
      P(\phi)=P^{(0)}(\phi)+\sum_{\upsilon=0}^{+\infty}(P^{(\upsilon+1)}(\phi)-P^{(\upsilon)}(\phi)).
      \end{equation}
\end{description}

%-------------------------------------------------------------------------------------------
\subsection{Lemmas}
We need the following Lemmas.

\begin{lemma} \cite{Bogoljubov1969}\label{lem7.2}
For $0<\delta<1, \nu>1,$ one has
$$\sum_{k\in \mathbb{Z}^{n}}e^{-2|k|\delta}|k|^{\nu}<\left(\frac{\nu}{e}\right)^{\nu}\frac{(1+e)^{n}}{\delta^{\nu+n}}
%\leq \frac{C}{\delta^{\nu+n}}
.$$
\end{lemma}
%%%%%%%%%%%%%%%%%%%%%%%%%%%%%%%%%%%%%%%%%%%%%%%%%

%%%%%%%%%%%%%%%%%%%%%%%%%%%%%%%%%%%%%%%%%%%%%%%%%

\begin{lemma}\cite{Poschel1996}\label{lem7.3}
If $A=(A_{ij})$ is a bounded linear operator on $\ell^{2},$ then also $B=(B_{ij})$ with
$$B_{ij}=\frac{|A_{ij}|}{|i-j|},\;i\neq j,$$
and $B_{ii}=0$ is bounded linear operator on $\ell^{2},$ and $\|B\|\leq \left(\frac{\pi}{\sqrt{3}}\right)\|A\|,$
where $\|\cdot\|$ is $\ell^{2}\rightarrow\ell^{2}$ operator norm.
\end{lemma}
%%%%%%%%%%%%%%%%%%%%%%%%%%%%%%%%%%%%%%%%%%%%%%%%%%%%%%

%%%%%%%%%%%%%%%%%%%%%%%%%%%%%%%%%%%%%%%%%%%%%
\begin{rem}\label{rem1}
Lemma \ref{lem7.3} holds true for the weight norm $\|\cdot\|_{N}.$
\end{rem}
%-----------------------------------------------------------------------------------------------------------
\section{Main results}

Consider the differential equation:
\begin{equation}\label{eq4}%----------------------eq(2.1)
\mathcal{L}u=\mi \,u_{t}-u_{xx}+Mu+\varepsilon W(\omega t, x)u=0
\end{equation}
subject to the boundary condition

\begin{equation}\label{eq5}%----------------------eq(2.2)
u(t,0)=u(t, \pi)=0.
\end{equation}
It is well-known that the Sturm-Liouville problem

\begin{equation}\label{eq6}%----------------------eq(2.3)
-y''+My=\lambda y,
\end{equation}
with the boundary condition

\begin{equation}\label{eq7}%----------------------eq(2.4)
y(0)=y(\pi)=0
\end{equation}
has the eigenvalues and eigenfunctions, respectively,

\begin{equation}\label{eq8}%----------------------eq(2.5)
\lambda_{k}=k^{2}+M,\;\;k=1, 2, \ldots,
\end{equation}
\begin{equation}\label{eq9}%----------------------eq(2.6)
\phi_{k}(x)=\sin kx,\;\;k=1, 2, \ldots .
\end{equation}
Write
\begin{equation}\label{eq10}%----------------------eq(2.7)
u(t, x)=\sum_{k=1}^{\infty} u_{k}(t)\phi_{k}(x).
\end{equation}
Note that $W$ is an even function of $x$. Write
\begin{equation}\label{eq11}%----------------------eq(2.8)
W(\omega t, x)=\sum_{k=0}^{\infty} v_{k}(\omega t)\varphi_{k}(x),
\end{equation}
where $$\varphi_{k}(x)=\cos 2kx,\;\;k=1, 2, \ldots .$$
Considering that
\begin{eqnarray*}
 W(\omega t, x)u(x)&=&\sum_{k=1}^{\infty}\langle W(\omega t, x)u(x), \phi_{k}(x)\rangle \phi_{k}\\
 &=& \sum_{k=1}^{\infty}\sum_{l=1}^{\infty}\sum_{j=0}^{\infty}c_{jlk}v_{j}u_{l}\phi_{k}(x),
\end{eqnarray*}
where
\begin{equation}\label{eq2.11-}
c_{jlk}=\int_{0}^{\pi} \varphi_{j}\phi_{l}\phi_{k}dx
=\int_{0}^{\pi}\cos 2jx\cdot\sin lx\cdot\sin kx \,dx=
\left\{%
\begin{array}{ll}
 \;\;0,\;k\neq \pm l\pm 2j,\\
\;\;\frac{\pi}{4},\; k= l\pm 2j \geq 1, \\
 -\frac{\pi}{4},\; k=-l\pm 2j\geq 1.\\
\end{array}%
\right.
\end{equation}
Then \eqref{eq4} can be expressed as
\begin{equation*}
\sum_{k=1}^{\infty}\left(\mi\,\dot{u}_{k}+\lambda_{k}u_{k}+\varepsilon \sum_{l=1}^{\infty}\sum_{j=0}^{\infty}c_{jlk}v_{j}u_{l}\right)\phi_{k}=0,
\end{equation*}
which implies that
\begin{equation}\label{eq13}%----------------------eq(2.10)
\mi\,\dot{u}_{k}+\lambda_{k}u_{k}+\varepsilon \sum_{l=1}^{\infty}\sum_{j=0}^{\infty}c_{jlk}v_{j}u_{l}=0.
\end{equation}
 This is a hamiltonian system
\begin{eqnarray}\label{eq19}%---------------------------------(2.17)
\left\{%
\begin{array}{cl}
&\mi\,\dot{u}_{k}=\frac{\partial H}{\partial \overline{u}_{k}},\;\;k\geq 1,\\
&\\
&\mi\,\dot{\overline{u}}_{k}=-\frac{\partial H}{\partial u_{k}},\;\;k\geq 1,\\
\end{array}%
\right.
\end{eqnarray}
with hamiltonian
\begin{equation}\label{eq20}%---------------------------------(2.18)
H(u,\overline{u})=\sum_{k=1}^{\infty}{\lambda_{k}}u_{k}
\overline{u}_{k}+\varepsilon\sum_{k=1}^{\infty}\sum_{l=1}^{\infty}\sum_{j=0}^{\infty}
c_{jlk}{v_{j}(\theta)}
u_{l}\overline{u}_{k}.
\end{equation}
For two sequences $x=(x_{j}\in \mathbb{C},\;j=1, 2,\ldots)$, $y=(y_{j}\in \mathbb{C},\;j=1, 2,\ldots),$
define
$$\langle x, y\rangle=\sum_{j=1}^{\infty}x_{j}y_{j}.$$
Then we can write
\begin{equation}\label{eq21}%---------------------------------(2.19)
H=\langle\Lambda u,\overline{u}\rangle+\varepsilon \langle R(\theta)u, \overline{u}\rangle,
\end{equation}
where
\begin{equation*}
\Lambda=diag \left({\lambda_{j}}: j=1,2,\ldots\right),\;\;\theta=\omega t,
\end{equation*}
\begin{equation}\label{eq22.1}
R(\theta)=\left(R_{kl}(\theta): k, l=1,2,\ldots\right),
\;\; R_{kl}(\theta)=\sum_{j=0}^{\infty}{c_{jlk}v_{j}(\theta)}.
\end{equation}%----------eq2.20
Define a Hilbert space $h_{N}$ as follows:
\begin{equation}\label{eq23}%------------------eq(2.23)
h_{N}=\{z=(z_{k}\in\mathbb{C}:k=1,2,\ldots)\}.
\end{equation}
Let $$\langle y, z\rangle_{N}:=\sum_{k=1}^{\infty}k^{2N}y_{k}\overline{z}_{k},\;\;\forall y, z\in h_{N}.$$
\begin{equation}\label{eq24}%------------------eq(2.24)
\|z\|_{N}^{2}=\langle z, z\rangle_{N}.
\end{equation}
Recall that $$\mathcal{W}(\theta,x)\in {C}^{N}(\mathbb{T}^{n}\times [0,\pi], \mathbb{R}).$$
Note that the Fourier transformation \eqref{eq10} is isometric from $u\in\mathcal{H}^{N}[0, \pi]$
to $(u_{k}: k=1, 2, \ldots)\in h_{N},$ where $\mathcal{H}^{N}[0, \pi]$ is the usual Sobolev space.
By \eqref{eq22.1}, we have that
\begin{eqnarray}\label{eq25}
\sup_{\theta\in\mathbb{T}^{n}}\|\sum_{|\alpha|= N}\partial^{\alpha}_{\theta} R(\theta) \|_{h_N\to h_N}\leq C,
\end{eqnarray}
where $ ||\cdot||_{h_N\to h_N}$ is the operator norm from $h_N$ to $h_N$, and $\alpha=(\alpha_{1}, \alpha_{2}, \ldots , \alpha_{n}),$ $|\alpha|=|\alpha_{1}|+|\alpha_{2}|+\ldots+|\alpha_{n}|,$ $\alpha_{j}' s\geq 0$
is an integer.

Actually,
$$\partial_{\theta}^{\alpha}R(\theta)=\left(\sum_{j=0}^{\infty}
C_{jlk}\partial_{\theta}^{\alpha}v_{j}(\theta): l, k=1, 2, \cdots\right).$$
For any $z=(z_{k}\in \mathbb{C}: k=1, 2, \cdots)\in h_{N},$
$$\left(\sum_{|\alpha|=N}\partial_{\theta}^{\alpha}R(\theta)\right)z=\left(\sum_{j=0}^{\infty}
\sum_{k=1}^{\infty}C_{jlk}(\sum_{|\alpha|=N}\partial_{\theta}^{\alpha}v_{j}(\theta))z_{k}: l=1, 2, \cdots\right).$$
Thus,
\begin{eqnarray}\label{eq2.25}
&&\left\|\left(\sum_{|\alpha|=N}\partial_{\theta}^{\alpha}R(\theta)\right)z\right\|_{N}^{2}\\\nonumber
&=&\sum_{l=1}^{\infty}l^{2N}\left| \sum_{j=0}^{\infty}
\sum_{k=1}^{\infty}C_{jlk}\left(\sum_{|\alpha|=N}\partial_{\theta}^{\alpha}v_{j}(\theta)\right)z_{k}\right|^{2}\\\nonumber
&=&\sum_{l=1}^{\infty}l^{2N}\left| \sum_{k=1}^{\infty}C_{0lk}\left(\sum_{|\alpha|=N}\partial_{\theta}^{\alpha}v_{0}(\theta)\right)z_{k}+
\sum_{j=1}^{\infty}
\sum_{k=1}^{\infty}C_{jlk}\left(\sum_{|\alpha|=N}\partial_{\theta}^{\alpha}v_{j}(\theta)\right)z_{k}\right|^{2}\\\nonumber
&\leq &C\sum_{l=1}^{\infty}l^{2N}\left(\left| \sum_{k=1}^{\infty}C_{0lk}\left(\sum_{|\alpha|=N}\partial_{\theta}^{\alpha}v_{0}(\theta)\right)z_{k}\right|^{2}+
\left| \sum_{j=1}^{\infty}
\sum_{k=1}^{\infty}C_{jlk}\left(\sum_{|\alpha|=N}\partial_{\theta}^{\alpha}v_{j}(\theta)\right)z_{k}\right|^{2}\right),
\end{eqnarray}
%%%%%%%%%%%%%%%%%%%%%%%%%%%%%%%%%%%%%%%%%%%%%%%%%%%%%%%%%%%
where
\begin{eqnarray}\label{eq2.251}
&&\sum_{l=1}^{\infty}l^{2N}\left| \sum_{k=1}^{\infty}C_{0lk}\left(\sum_{|\alpha|=N}\partial_{\theta}^{\alpha}v_{0}(\theta)\right)z_{k}\right|^{2}\\\nonumber
&\leq & C\left| \sum_{|\alpha|=N}\partial_{\theta}^{\alpha}v_{0}(\theta)\right|^{2}\sum_{l=1}^{\infty}l^{2N} |z_{l}|^{2}
\leq C \|z\|^{2}_{N},
\end{eqnarray}
here we used the fact that
$C_{0lk}=0$, if $k\neq l.$
%%%%%%%%%%%%%%%%%%%%%%%%%%%%%%%%%%%%%%%%%%%
Let $$\gamma_{\,lj}=\frac{(\pm l\pm j)j}{l},\;\;\mbox{where}\;\;(\pm l\pm j)jl\neq 0.$$
Thus,
\begin{eqnarray}\label{eq2.252}
&&\left\|\left(\sum_{|\alpha|=N}\partial_{\theta}^{\alpha}JR^{zz}(\theta)J\right)z\right\|_{N}^{2}\\\nonumber
&=&\sum_{l=1}^{\infty}l^{2N}\left| \frac{1}{2}\sum_{j=1}^{\infty}
\sum_{k=1}^{\infty}C_{jlk}(\sum_{|\alpha|=N}\partial_{\theta}^{\alpha}v_{j}(\theta))z_{k}\right|^{2}\\\nonumber
&=&\sum_{l=1}^{\infty}l^{2N}\left|\frac{1}{2}\sum_{j=1}^{\infty} C_{jl(\pm l\pm j)}(\sum_{|\alpha|=N}\partial_{\theta}^{\alpha}v_{j}(\theta))z_{\pm l\pm j}\right|^{2}\\\nonumber
&=&\frac{1}{4}\sum_{l=1}^{\infty}l^{2N}\left|\frac{1}{\gamma_{\,lj}^{N}}\cdot \gamma_{\,lj}^{N}\sum_{j=1}^{\infty} C_{jl(\pm l\pm j)}(\sum_{|\alpha|=N}\partial_{\theta}^{\alpha}v_{j}(\theta))z_{\pm l\pm j}\right|^{2}\\\nonumber
&\leq & C\left(\sum_{j=1}^{\infty}\frac{1}{\gamma_{\,lj}^{2N}}\right)\left(\sum_{j=1}^{\infty}|j|^{2N}\Big|\sum_{|\alpha|=N}\partial_{\theta}^{\alpha}v_{j}(\theta)
\Big|^{2}\sum_{l=1}^{\infty}|\pm l\pm j|^{2N}|z_{\pm l\pm j}|^{2}\right)\\\nonumber
&\leq &C\sum_{j=1}^{\infty}|j|^{2N}\Big|\sum_{|\alpha|=N}\partial_{\theta}^{\alpha}v_{j}(\theta)\Big|^{2}\|z\|_{N}^{2}\\\nonumber
&\leq &C\sup_{(\theta,x)\in \mathbb{T}^{n}\times [-\pi, \pi]}
\big|\sum_{|\alpha|=N}\partial_{\theta}^{\alpha}\partial_{x}^{N}\mathcal{V}(\theta, x)\big|\|z\|_{N}^{2}\leq C \|z\|_{N}^{2},
\end{eqnarray}
where $C$ is a universal constant which might be different in different places.
%%%%%%%%%%%%%%%%%%%%%%%%%%%%%%%%%%%%%%%%%%%%%%%%%%%%%%%%%%%%%%%%%%%%%%%%%%%%%%%%%%%%%%%%%%%%%%%%
Combing \eqref{eq2.25}, \eqref{eq2.251} and \eqref{eq2.252}, we have
\begin{equation}\label{eq2.26}
\sup_{\theta\in\mathbb{T}^{n}}\|\sum_{|\alpha|= N}\partial^{\alpha}_{\theta}R(\theta)\|_{h_N\to h_N}\leq C.
\end{equation}

%%%%%%%%%%%%%%%%%%%%%%%%%%%%%%%%%%%%%%%%%%%%%%%%%%%%%%%%%%%%%%%%%%%%%%%%%%%%%%%%%%%%%%%%%%%%%%%%%%%%%%%%%%%%%%%%%%%%%%%%%%%%%%%%%%%%%%%%%%
 Now let us apply analytical approximation Lemma to the perturbation $P(\phi).$
Take a sequence of real numbers $\{s_{v}\geq 0\}_{v=0}^{\infty}$ with $s_{v}>s_{v+1}$ goes fast to zero.
Let $R(\theta)=P(\theta)$. Then by \eqref{cite3.6}
 we can write,
  \begin{equation}\label{*}
 R(\theta)=R_{0}(\theta)+\sum_{l=1}^{\infty}R_{l}(\theta),
 \end{equation}
 where $R_{0}(\theta)$ is analytic in $\mathbb{T}_{s_{0}}^{n}$ with
  \begin{equation}\label{**}
 \sup_{\theta\in\mathbb{T}^{n}_{s_{0}}}\|R_{0}(\theta)\|_{h_{N}\to h_{N}}\leq C,
  \end{equation}
 and $R_{l}(\theta)\;(l\geq 1)$ is analytic in $\mathbb{T}^{n}_{s_{l}}$ with
 \begin{equation}\label{***}
 \sup_{\theta\in\mathbb{T}^{n}_{s_{l}}}\|R_{l}(\theta) \|_{h_{N}\to h_{N}}\leq Cs^{N}_{l-1}.
  \end{equation}
 %Because $N$ is fixed. In the following, we write $\|\cdot\|=\|\cdot\|_{N}.$
 %%%%%%%%%%%%%%%%%%%%%%%%%%%%%%%%%%%%%%%%%%%%%%%%%%%%%%%%%%%%%%%%%%%%%%%%%%%%%%%%%%%%%%%%%%%%%%%
 \subsection{Iterative parameters of domains}
 Let

\begin{itemize}
  \item $\varepsilon_{0}=\varepsilon, \varepsilon_{\nu}=\varepsilon^{(\frac{4}{3})^{\nu}}, \nu=0, 1, 2, \ldots,$
  which measures the size of perturbation at $\nu-th$ step.
  \item $s_{\nu}=\varepsilon_{\nu+1}^{1/N}, \nu=0, 1, 2, \ldots,$
  which measures the strip-width of the analytic domain $\mathbb{T}_{s_{\nu}}^{n},$
  $\mathbb{T}_{s_{\nu}}^{n}=\{\theta\in \mathbb{C}^{n}/(\pi \mathbb{Z})^{n}: |Im \theta|\leq s_{\nu}\}.$
  \item $C({\nu})$ is a constant which may be different in different places, and it is of the form
   $$C({\nu})= C_{1} 2^{C_{2}\nu},$$
  where $C_{1},$ $C_{2}$ are absolute(?) constants.
  \item $K_{\nu}=10 s_{\nu}^{-1}(\frac{4}{3})^{\nu}|\log \varepsilon|.$
  \item$\gamma_{\nu}=\frac{\gamma}{2^{\nu}},\,0<\gamma\ll 1.$
  \item a family of subsets $\Pi_{\nu}\subset [1, 2]$ with $[1, 2]\supset \Pi_{0}\supset \ldots \supset \Pi_{\nu}\supset \ldots,$
  and $$mes \Pi_{\nu}\geq mes \Pi_{\nu-1}-C\gamma_{\nu-1},$$

  \item For an operator-value (or a vector function) $B(\theta, \tau),$ whose domain is
  $(\theta, \tau)\in \mathbb{T}_{s_{\nu}}^{n}\times \Pi_{\nu}.$ Set
  $$ \|B\|_{\mathbb{T}^{n}_{s_{\nu}}\times \Pi_{\nu}}=\sup_{(\theta, \tau)
  \in \mathbb{T}^{n}_{s_{\nu}}\times \Pi_{\nu}}\|B(\theta, \tau)\|_{h_{N}\to h_{N}},$$ where $\|\cdot\|_{h_{N}\to h_{N}}$ is the operator norm, and set
  $$ \|B\|^{\mathcal{L}}_{\mathbb{T}^{n}_{s_{\nu}}\times \Pi_{\nu}}=\sup_{(\theta, \tau)
  \in \mathbb{T}^{n}_{s_{\nu}}\times \Pi_{\nu}}\|\partial_{\tau}B(\theta, \tau)\|_{h_{N}\to h_{N}}.$$
\end{itemize}
%%%%%%%%%%%%%%%%%%%%%%%%%%%%%%%%%%%%%%%%%%%%%%%%%%%%%%%%%%%%%%%%%%%%%%%%%%%%%%%%%%%%%%%
\subsection{Iterative Lemma}
In the following, for a function $f(\omega)$, denote by $\partial_{\omega}$ the derivative of $f(\omega)$ with respect to $\omega$ in Whitney's sense.
\begin{lemma}
Let $R_{0,0}=R_{0}, R_{l,0}=R_{l},$ where
$R_{0}, R_{l}$ are defined by \eqref{*}.
Assume that we have a family of Hamiltonian functions $H_{\nu}$:
\begin{equation}\label{eq26}%------------------------eq(5.1)
H_{\nu}=\sum_{j=1}^{\infty}\lambda_{j}^{(\nu)}u_{j}\overline{u}_{j}
+\sum_{l\geq \nu}^{\infty}\varepsilon_{l}\langle R_{l,\nu}u,\overline{u}\rangle,\;\nu=0,1, \ldots, m,
\end{equation}
where $R_{l,\nu}$ is operator-valued function defined on the domain $\mathbb{T}_{s_{\nu}}^{n}\times \Pi_{\nu},$ and
\begin{equation}\label{eq27}%------------------------eq(5.2)
\theta=\omega t.
\end{equation}
\begin{description}
  \item[$(A1)_{\nu}$]
  \begin{equation}\label{eq28}%------------------------eq(5.3)
  \lambda_{j}^{(0)}={\lambda_{j}}=j^{2}+M,\;
  \lambda_{j}^{(\nu)}={\lambda_{j}}+\sum_{i=0}^{\nu-1}\varepsilon_{i}\mu_{j}^{(i)},\;\;\nu\geq 1
  \end{equation}
and $\mu_{j}^{(i)}=\mu_{j}^{(i)}(\tau):\Pi_{i}\rightarrow\mathbb{R}$ with
\begin{eqnarray}%------------------------eq(5.4)
 &&|\mu_{j}^{(i)}|_{\Pi_{i}}:=\sup_{\tau\in\Pi_{i}}|\mu_{j}^{(i)}(\tau)|\leq C(i),\;
  0\leq i\leq\nu-1,\label{eq29}\\ &&|\mu_{j}^{(i)}|_{\Pi_{i}}^{\mathcal{L}}:=\sup_{\tau\in\Pi_{i}}|\partial_{\tau}\mu_{j}^{(i)}(\tau)|\leq C(i),\;
  0\leq i\leq\nu-1.\label{eq029}
  \end{eqnarray}
 \item[$(A2)_{\nu}$]
  $R_{l,\nu}=R_{l,\nu}(\theta, \tau)$ is defined in
  $\mathbb{T}^{n}_{s_{l}}\times \Pi_{\nu}$ with $l\geq \nu,$ and is analytic in $\theta$ for fixed $\tau\in \Pi_{\nu},$
  and
  \begin{equation}\label{eq30}%------------------------eq(5.5)
  \|R_{l,\nu} \|_{\mathbb{T}^{n}_{s_{l}}\times \Pi_{\nu}}\leq C(\nu),
  \end{equation}
  \begin{equation}\label{eq31}%------------------------eq(5.6)
  \|R_{l,\nu} \|^{\mathcal{L}}_{\mathbb{T}^{n}_{s_{l}}\times \Pi_{\nu}}\leq C(\nu).
  \end{equation}
\end{description}
Then there exists a compact set $\Pi_{m+1}\subset \Pi_{m}$ with
 \begin{equation}\label{eq32}%------------------------eq(5.7)
 mes \Pi_{m+1}\geq mes \Pi_{m}-C\gamma_{m},
  \end{equation}
  and exists a symplectic coordinate changes
  \begin{equation}\label{eq33}%------------------------eq(5.8)
\Psi_{m}: \mathbb{T}^{n}_{s_{m+1}}\times \Pi_{m+1}\rightarrow \mathbb{T}^{n}_{s_{m}}\times \Pi_{m},
  \end{equation}
  \begin{equation}\label{eq33H}%------------------------eq(5.8H)
||\Psi_{m}-id ||_{h_{N}\to h_{N}}\leq \varepsilon^{1/2},\;(\theta, \tau)\in \mathbb{T}^{n}_{s_{m+1}}\times \Pi_{m+1}
  \end{equation}
  such that the Hamiltonian function $H_{m}$ is changed into
 \begin{eqnarray}\label{eq34}%------------------------eq(5.9)
H_{m+1}& \triangleq & H_{m}\circ \Psi_{m}\nonumber\\
&=&\sum_{j=1}^{\infty}\lambda_{j}^{(m+1)}u_{j}\overline{u}_{j}
+\sum_{l\geq m+1}^{\infty}\varepsilon_{l}\langle R_{l,m+1}u,\overline{u}\rangle,
  \end{eqnarray}
  which is defined on the domain $\mathbb{T}^{n}_{s_{m+1}}\times \Pi_{m+1},$
  and ${\lambda_{j}^{(m+1)}}^{,}s$ satisfy the assumptions $(A 1)_{m+1}$ and
  $R_{l, m+1}$ satisfy the assumptions $(A 2)_{m+1}.$
\end{lemma}
%%%%%%%%%%%%%%%%%%%%%%%%%%%%%%%%%%%%%%%%%%%%%%%%%%%%%%%%%%%%%%%%%%%%%%%%%%%%%%%%%%%%%%%%%%%%%%%%%
\subsection{Derivation of homological equations}
Our end is to find a symplectic transformation $\Psi_{\nu}$ such that the terms ${R_{l, v}}$
(with $l=v$) disappear. To this end, let $F$ be a linear Hamiltonian of the form
 \begin{equation}\label{eq35}%------------------------eq(5.10)
F=\langle F(\theta, \tau)u, \overline{u}\rangle,
  \end{equation}
where $\theta=\omega t,$ $(F(\theta, \tau))^{T}=F(\theta, \tau).$
Moreover, let
\begin{equation}\label{eq36}%------------------------eq(5.11)
\Psi=\Psi_{m}=X_{\varepsilon_{m}F}^{t}\big | _{t=1},
 \end{equation}
 where $X^{t}_{\varepsilon_{m}F}$ is the flow of the Hamiltonian. Vector field $X_{\varepsilon_{m}F}$
 of the Hamiltonian $\varepsilon_{m}F$ with the symplectic structure $\sqrt{-1} du\Lambda d\overline{u}.$
 Let
 \begin{equation}\label{eq37}%------------------------eq(5.12)
H_{m+1}=H_{m}\circ \Psi_{m}.
 \end{equation}
 By \eqref{eq26}, we write
 \begin{equation}\label{eq38}%------------------------eq(5.14)
H_{m}=N_{m}+R_{m},
 \end{equation}
 with
\begin{equation}\label{eq39}%------------------------eq(5.15)
N_{m}=\sum_{j=1}^{\infty}\lambda_{j}^{(m)}u_{j}\overline{u}_{j},
 \end{equation}
 \begin{equation}\label{eq40}%------------------------eq(5.16)
R_{m}=\sum_{l=m}^{\infty}\varepsilon_{l}R_{lm}
=\sum_{l=m}^{\infty}\varepsilon_{l}\langle R_{l,m}(\theta)u, \overline{u}\rangle,
 \end{equation}
 where $(R_{l,m}(\theta))^{T}=R_{l,m}(\theta).$
 Since the Hamiltonian $H_{m}=H_{m}(\omega t, u, \overline{u})$ depends on time $t,$ we introduce a fictitious
 action $I=$ constant, and let $\theta=\omega t$ be angle variable. Then the non-autonomous $H_{m}(\omega t, u, \overline{u})$ can be written as
 $$\omega I+H_{m}(\theta, u, \overline{u})$$
 with symplectic structure $d I\Lambda d\theta +\sqrt{-1} d u\Lambda d \overline{u}$.
 By combination of \eqref{eq35}-\eqref{eq40} and Taylor formula, we have
 \begin{eqnarray}\label{eq42}%------------------------eq(5.18)
H_{m+1}&=&H_{m}\circ X^{1}_{\varepsilon_{m}F}%=(\omega I+H_{m})\circ X^{1}_{\varepsilon_{m}F}
\nonumber\\
&=&N_{m}+\varepsilon_{m}\{N_{m},F\}
+\varepsilon^{2}_{m}\int_{0}^{1}(1-\tau)\{\{N_{m}, F\}, F\}\circ X^{\tau}_{\varepsilon_{m}F}d\tau
+\varepsilon_{m}\omega\cdot\partial_{\theta}F\nonumber\\
&&+\varepsilon_{m}R_{mm}+(\sum_{l=m+1}^{\infty}\varepsilon_{l}R_{lm})\circ X^{1}_{\varepsilon_{m}F}
+\varepsilon_{m}^{2}\int_{0}^{1}\{R_{mm}, F\}\circ X^{\tau}_{\varepsilon_{m} F}d\tau,
 \end{eqnarray}
 where $\{\cdot, \cdot\}$ is the Poisson bracket with respect to $\sqrt{-1} du\Lambda d \overline{u},$
  that is $$\{H(u, \overline{u}), F(u, \overline{u})\}=
  -\sqrt{-1}\left(\frac{\partial H}{\partial u}\cdot\frac{\partial F}{\partial \overline{u}}
  -\frac{\partial H}{\partial \overline{u}}\cdot\frac{\partial F}{\partial u}\right).$$
 Let $\Gamma_{K_{m}}$ be a truncation operator.
For any
$$
 f(\theta)=\sum_{k\in \mathbb{Z}^{n}}\widehat{f}(k)e^{i\langle k, \theta\rangle},\;\theta\in\mathbb{T}^{n}.
$$
 Define, for given $K_{m}>0,$
 $$\Gamma_{K_{m}} f(\theta)=(\Gamma_{K_{m}} f)(\theta)\triangleq \sum_{|k|\leq K_{m}}\widehat{f}(k)e^{i\langle k, \theta\rangle},$$
 $$(1-\Gamma_{K_{m}}) f(\theta)=((1-\Gamma_{K_{m}}) f)(\theta)\triangleq \sum_{|k|> K_{m}}\widehat{f}(k)e^{i\langle k, \theta\rangle}.$$
 Then
 $$f(\theta)=\Gamma_{K_{m}} f(\theta)+(1-\Gamma_{K_{m}})f(\theta).$$
 Let
 \begin{equation}\label{eq43}%------------------------eq(5.19)
\omega\cdot\partial_{\theta}F+\{N_{m}, F\}+\Gamma_{K_{m}}R_{mm}=\langle[R_{mm}]u, \overline{u}\rangle,
  \end{equation}
  where
 \begin{equation}\label{eq44}%------------------------eq(5.20)
  [R_{mm}]:=diag \left(\widehat{R}_{mmjj}(0): j=1, 2, \ldots\right),
  \end{equation}
  and $R_{mmij}(\theta)$ is the matrix element of $R_{mm}$ and $\widehat{R}_{mmij}(k)$ is the
  $k$-Fourier coefficient of $R_{mmij}(\theta).$
  Then
  \begin{equation}\label{eq45}%------------------------eq(5.21)
  H_{m+1}=N_{m+1}+C_{m+1}R_{m+1},
  \end{equation}
  where
  \begin{equation}\label{eq46}%------------------------eq(5.22)
 N_{m+1}=N_{m}+\varepsilon_{m}\langle [R_{mm}]u, \overline{u}\rangle
 =\sum_{j=1}^{\infty}\lambda_{j}^{(m+1)}u_{j}\overline{u}_{j},
  \end{equation}
\begin{equation}\label{eq5.23}%------------------------eq(5.23)
\lambda_{j}^{(m+1)}=\lambda_{j}^{(m)}+\varepsilon_{m}\widehat{R}_{mmjj}(0)
={\lambda_{j}}+\sum_{l=1}^{m}\varepsilon_{l}\mu_{j}^{(l)},
\;\mu_{j}^{(m)}:=\widehat{R}_{mmjj}(0).
  \end{equation}
  \begin{eqnarray}
 C_{m+1}R_{m+1}&=&\varepsilon_{m}(1-\Gamma_{K_{m}})R_{mm}\label{eq5.24}\\%------------------------eq(5.24)
 &+&\varepsilon_{m}^{2}\int_{0}^{1}(1-\tau)\{\{N_{m}, F\},F\}\circ X_{\varepsilon_{m}F}^{\tau}d\tau\label{eq5.27}\\%------------------------eq(5.27)
 &+&\varepsilon_{m}^{2}\int_{0}^{1}\{R_{mm}, F\}\circ X^{\tau}_{\varepsilon_{m} F}d\tau\label{eq5.28}\\%------------------------eq(5.28)
  &+&\left(\sum_{l=m+1}^{\infty}\varepsilon_{l}R_{lm}\right)\circ X_{\varepsilon_{m}F}^{1}.\label{eq5.25}%------------------------eq(5.25)
  \end{eqnarray}
  The equation \eqref{eq43} is called homological equation. Developing the Poisson bracket $\{N_{m},F\}$ and comparing
  the coefficients of $u_{i}\overline{u}_{j} (i, j=1, 2, \ldots),$ we get
  \begin{equation}\label{eq48}
 \omega\cdot\partial_{\theta} F(\theta, \tau)
       +\sqrt{-1}(F(\theta, \tau)\Lambda^{(m)}-\Lambda^{(m)}F(\theta, \tau))
=\Gamma_{K_{m}}R_{mm}(\theta)-[R_{mm}],
\end{equation}
where
\begin{equation}\label{eq49}%------------------------eq(5.30)
\Lambda^{(m)}=diag(\lambda_{j}^{(m)}: j=1, 2, \ldots),
\end{equation}
and we assume
$\Gamma_{K_{m}}F(\theta, \tau)=F(\theta, \tau).$
Write $F_{ij}(\theta)$ is the matrix elements of $F(\theta, \tau)$ .
Then \eqref{eq48} can be rewritten as:
    \begin{equation}\label{eq5.33}%------------------------eq(5.33)
\omega\cdot\partial_{\theta} F_{ij}(\theta)
       -\sqrt{-1}(\lambda_{i}^{(m)}-\lambda_{j}^{(m)})F_{ij}(\theta)=
       \Gamma_{K_{m}}R_{mmij}(\theta),\;
i\neq j,
\end{equation}%------------------------eq(5.33)
 \begin{equation}\label{eq5.34}%------------------------eq(5.34)
\omega\cdot\partial_{\theta} F_{ii}(\theta)=\Gamma_{K_{m}}R_{mmii}(\theta)-\widehat{R}_{mmii}(0),%------------------------eq(5.34)
\end{equation}
where $i,j=1, 2, \ldots.$
%%%%%%%%%%%%%%%%%%%%%%%%%%%%%%%%%%%%%%%%%%%%%%%%%%%%%%%%%%%%%%%%%%%%%%%%%%%%%%%%%%%%%%%%%%%%%%%%%%
\subsection{Solutions of the homological equations}
\begin{lemma}\label{lem7.1}
There exists a compact subset $\Pi_{m+1}\subset \Pi_{m}$ with
\begin{equation}\label{eq7.1}%------------------------eq(7.1)
mes(\Pi_{m+1})\geq mes \Pi_{m}-C\gamma_{m}
\end{equation}
such that for any $\tau\in \Pi_{m+1}$ (Recall $\omega=\tau\omega_{0}$),
the equation \eqref{eq5.33} has a unique solution $F(\theta, \tau),$
which is defined on the domain $\mathbb{T}_{s_{m+1}}^{n}\times \Pi_{m+1},$
with
\begin{equation}\label{eq7.2}%------------------------eq(7.2)
\|F(\theta,\tau)\|_{\mathbb{T}_{s_{m+1}}^{n}\times
\Pi_{m+1}}\leq C(m+1)\varepsilon_{m}^{-\frac{2(2n+3)}{N}},
\end{equation}
\begin{equation}\label{eq7.3}%------------------------eq(7.3)
\|F(\theta,\tau)\|^{\mathcal{L}}_{\mathbb{T}_{s_{m+1}}^{n}\times
\Pi_{m+1}}\leq C(m+1)\varepsilon_{m}^{-\frac{4(2n+3)}{N}}.
\end{equation}
\end{lemma}
%%%%%%%%%%%%%%%%%%%proof%%%%%%%%%%%%%%%%%%%%%%%%%%
\begin{proof}
By passing to Fourier coefficients, \eqref{eq5.33} can be rewritten as
\begin{equation}\label{eq7.4}
(-\langle k, \omega\rangle+\lambda_{i}^{(m)}-\lambda_{j}^{(m)})\widehat{F}_{ij}(k)
=\sqrt{-1}\widehat{R}_{mmij}(k),
\end{equation}
where $i, j=1, 2, \ldots, k\in \mathbb{Z}^{n}$ with $|k|\leq K_{m}.$ Recall $\omega=\tau \omega_{0}.$
Let
$$A_{k}=|k|^{n+3},$$
and let
\begin{equation}\label{eq7.5}
Q_{kij}^{(m)}\triangleq \left\{\tau\in\Pi_{m}\bigg| \big|-\langle k, \omega_{0}\rangle\tau
+\lambda_{i}^{(m)}-\lambda_{j}^{(m)}\big|<\frac{(|i-j|+1)\gamma_{m}}{A_{k}}\right\},
\end{equation}
where $i, j=1, 2, \ldots,$ $k\in\mathbb{Z}^{n}$ with $|k|\leq K_{m},$ and $k\neq 0$ when $i=j.$
Let
\begin{equation*}%\label{eq7.6}
\Pi_{m+1}=\Pi_{m}\diagdown\bigcup_{|k|\leq K_{m}}\bigcup_{i=1}^{\infty}\bigcup_{j=1}^{\infty}Q_{kij}^{(m)}.
\end{equation*}
Then for any $\tau\in\Pi_{m+1},$ we have
\begin{equation}\label{eq7.6}
\big|-\langle k, \omega\rangle
+\lambda_{i}^{(m)}-\lambda_{j}^{(m)}\big|\geq\frac{(|i-j|+1)\gamma_{m}}{A_{k}}.
\end{equation}
Recall that $R_{mm}(\theta)$ is analytic in the domain $\mathbb{T}_{s_{m}}^{n}$
for any $\tau\in \Pi_{m},$
\begin{equation}\label{eq7.8}
|\widehat{R}_{mmij}(k)|
\leq C(m)e^{-s_{m}|k|}.
\end{equation}
It follows
\begin{eqnarray}\label{eq7.7}
|\widehat{F}_{ij}(k)|
&=&\Bigg|\frac{\widehat{R}_{mmij}(k)}{-\langle k, \omega\rangle +\lambda_{i}^{(m)}-\lambda_{j}^{(m)}}\Bigg|\leq \frac{A_{k}}{\gamma_{m}(|i-j|+1)}\cdot|\widehat{R}_{mmij}(k)|\nonumber\\
&\leq &
\frac{|k|^{n+3}}{\gamma_{m}(|i-j|+1)}\cdot C(m)e^{-s_{m}|k|}.
\end{eqnarray}
%where $|i-j|$ takes $1$ when $i=j.$

Therefore, by \eqref{eq7.7}, we have
\begin{eqnarray*}
&&\sup_{\theta\in\mathbb{T}^{n}_{s'_{m}}\times \Pi_{m+1}}|F_{ij}(\theta, \tau)|\\
&\leq &\frac{C(m)}{\gamma_{m}(|i-j|+1)}\sum_{|k|\leq K_{m}}|k|^{n+3}e^{-(s_{m}-s'_{m})|k|}\\
&\leq & \left(\frac{n+3}{e}\right)^{n+3}\!(1+e)^{n}\left(\frac{2}{s_{m}-s'_{m}}\right)^{2n+3}\!\cdot
\frac{C(m)}{\gamma_{m}(|i-j|+1)}
\;( \mbox{by Lemma 7.2})\\
&\leq & \frac{1}{(s_{m}-s'_{m})^{2n+3}}\cdot\frac{C\cdot C(m)}{\gamma_{m}(|i-j|+1)}\\
&\leq & \varepsilon_{m}^{-\frac{2(2n+3)}{N}}\cdot\frac{C\cdot C(m)}{\gamma_{m}(|i-j|+1)},
\end{eqnarray*}
where $C$ is a constant depending on $n,$ $s'_{m}=s_{m}-\frac{s_{m}-s_{m+1}}{4}.$

By Lemma \ref{lem7.3} and the Remark \ref{rem1}, we have
\begin{equation}\label{eq7.9}
\|F(\theta, \tau)\|_{\mathbb{T}^{n}_{s'_{m}}\times \Pi_{m+1}}
\leq C\cdot C(m)\gamma_{m}^{-1}\varepsilon_{m}^{-\frac{2(2n+3)}{N}}
\leq C(m+1)\varepsilon_{m}^{-\frac{2(2n+3)}{N}}.
\end{equation}
It follows $s'_{m}>s_{m+1}$ that
$$\|F(\theta, \tau)\|_{\mathbb{T}^{n}_{s_{m+1}}\times \Pi_{m+1}}
\leq\|F(\theta, \tau)\|_{\mathbb{T}^{n}_{s'_{m}}\times \Pi_{m+1}}
\leq C(m+1)\varepsilon_{m}^{-\frac{2(2n+3)}{N}}.$$
Applying $\partial_{\tau}$ to both sides of \eqref{eq7.4}, we have
\begin{equation}\label{eq7.10}
\left(-\langle k, \omega\rangle +\lambda_{i}^{(m)}-\lambda_{j}^{(m)}\right)\partial_{\tau}\widehat{F}_{ij}(k)
=\sqrt{-1}\partial_{\tau}\widehat{R}_{mmij}(k)+(*),
\end{equation}
where
\begin{equation}\label{eq7.11}
(*)=-\left(-\langle k, \omega_{0}\rangle+\partial_{\tau}(\lambda_{i}^{(m)}-\lambda_{j}^{(m)})\right)\widehat{F}_{ij}(k).
\end{equation}
Recalling $|k|\leq K_{m}=10s_{m}^{-1}(\frac{4}{3})^{m}|\log \varepsilon|,$ and using \eqref{eq28} and \eqref{eq29} with $\nu=m,$ and
using \eqref{eq7.9}, we have, on $\tau\in\Pi_{m+1},$
\begin{equation}\label{eq7.12}
|(*)|\leq C\cdot C(m)\gamma_{m}^{-1}\varepsilon_{m}^{-\frac{2(2n+3)}{N}}e^{-s'_{m}|k|}.
\end{equation}
According to \eqref{eq31},
\begin{equation}\label{eq7.13}
\mid \partial _{\tau}\widehat{R}_{mmij}(k)\mid \leq C(m)e^{-s_{m}|k|}.
\end{equation}
By \eqref{eq7.6}, \eqref{eq7.10}, \eqref{eq7.12} and \eqref{eq7.13}, we have
\begin{equation}\label{eq7.14}
\mid  \partial _{\tau}\widehat{F}_{ij}(k)\mid
\leq \frac{A_{k}}{\gamma_{m}(|i-j|+1)}\cdot C\cdot C(m)\gamma_{m}^{-1}\varepsilon_{m}^{-\frac{2(2n+3)}{N}}e^{-s'_{m}|k|} \;\;\mbox{for}\;\; i\neq j.
\end{equation}
Note that $s_{m}>s'_{m}>s_{m+1}.$ Again using Lemma \ref{lem7.2} and Lemma\ref{lem7.3}, we have
\begin{eqnarray}\label{eq7.15}
&&\|F(\theta, \tau)\|^{\mathcal{L}}_{\mathbb{T}_{s_{m+1}}\times \Pi_{m+1}}
=\|  \partial _{\tau}{F}(\theta, \tau)\|_{\mathbb{T}_{s_{m+1}}\times \Pi_{m+1}}
\nonumber\\
&\leq &C^{2}\cdot C(m)\gamma^{-2}_{m}\varepsilon_{m}^{-\frac{4 (2n+3)}{N}}\leq C(m+1)\varepsilon_{m}^{-\frac{4 \,(2n+3)}{N}}.
\end{eqnarray}
The proof of the measure estimate \eqref{eq7.1} will be postponed to section \ref{em}.
This completes the proof of Lemma \ref{lem7.1}.
\end{proof}

%%%%%%%%%%%%%%%%%%%%%%%%%%%%%%%%%%%secton 8%%%%%%%%%%%%%%%%%%%%%%%%%%%%%%%%%%%%%%%%%%%%%%%%%%%%%%%%%%%%%%%%%%%
\subsection{Coordinate change $\Psi$ by $\varepsilon_{m} F$}
Recall $\Psi=\Psi_{m}=X_{\varepsilon_{m}F}^{t}\big| _{t=1},$ where $X^{t}_{\varepsilon_{m}F}$ is the flow of the Hamiltonian $\varepsilon_{m} F$, vector field $X_{\varepsilon_{m}F}$ with symplectic $\sqrt{-1}du\Lambda d\overline{u}.$
So
$$\sqrt{-1} \dot{u}=\varepsilon_{m}\frac{\partial F}{\partial \overline{u}},\;
-\sqrt{-1}\dot{\overline{u}}=\varepsilon_{m}\frac{\partial F}{\partial {u}},\;\dot{\theta}=\omega.$$
More exactly,
$$\left\{
    \begin{array}{ll}
      \sqrt{-1}\dot{u}=\varepsilon_{m}F(\theta, \tau)u,\;\theta=\omega t,\\
     -\sqrt{-1}\dot{\overline{u}}=\varepsilon_{m}F(\theta, \tau)\overline{u},\;\theta=\omega t, \\
 \dot{\theta}=\omega.
    \end{array}
  \right.
$$

Let $z=\left(
         \begin{array}{c}
           u \\
           \overline{u} \\
         \end{array}
       \right),$
      %%%%%%%%%%%%%%%%%%%%%%%%%%%%%%%%%%%%%%%%%%%%%%%%%%%%%%%%%%%%%%%%%%%%%%%%%%%
\begin{equation}\label{eq8.0}
B_{m}=\left(
            \begin{array}{cc}
            -\sqrt{-1} F(\theta, \tau) & 0\\
             0 & \sqrt{-1} F(\theta, \tau)\\
            \end{array}
          \right).\;\;\mbox{Recall }\;\;\theta=\omega t.
\end{equation}
%where $id$ is the identity from $h_{N}\to h_{N}.$
%%%%%%%%%%%%%%%%%%%%%%%%%%%%%%%%%%%%%%%%%%%%%%%%%%%%%%%%%%%%%%%%%%%%%%%%%%%%%%%%%%%%%%%%%%%%
Then
\begin{equation}\label{eq8.1}
\frac{dz(t)}{dt}=\varepsilon_{m}B_{m}(\theta)z,\;\;\dot{\theta}=\omega.
\end{equation}
Let $z(0)=z_{0}\in h_{N}\times h_{N},\,\theta(0)=\theta_{0}\in \mathbb{T}^{n}_{s_{m+1}}$ be
initial value.
Then
\begin{equation}\label{eq8.2}
\left\{
  \begin{array}{ll}
    z(t)=z_{0}+\int_{0}^{t}\varepsilon_{m}B_{m}(\theta_{0}+\omega s)z(s)ds, \\
    \theta(t)=\theta_{0}+\omega t.
  \end{array}
\right.
\end{equation}
By Lemmas \ref{lem7.1} in Section 7,
\begin{equation}\label{eq8.3}
\|B_{m}(\theta)\|_{\mathbb{T}^{n}_{s_{m+1}}\times \Pi_{m+1}}\leq C(m+1)\varepsilon_{m}^{-\frac{2(2n+3)}{N}},
\end{equation}
\begin{equation}\label{eq8.4}
\|B_{m}(\theta)\|^{\mathcal{L}}_{\mathbb{T}^{n}_{s_{m+1}}\times \Pi_{m+1}}\leq C(m+1)\varepsilon_{m}^{-\frac{4(2n+3)}{N}}.
\end{equation}
It follows from \eqref{eq8.2} that
$$z(t)-z_{0}=\int_{0}^{t}\varepsilon_{m}B_{m}(\theta_{0}+\omega s)z_{0}ds+\int_{0}^{t}\varepsilon_{m}B_{m}(\theta_{0}+\omega s)(z(s)-z_{0})ds.$$
Moreover, for $t\in [0,1]$, $\|z_{0}\|_{N}\leq 1,$
\begin{equation}\label{eq8.5}
\|z(t)-z_{0}\|_{N}\leq \varepsilon_{m}C(m+1)\varepsilon_{m}^{-\frac{4(2n+3)}{N}}
+\int_{0}^{t}\varepsilon_{m}\|B_{m}(\theta_{0}+\omega s)\|\|z(s)-z_{0}\|_{N}ds,
\end{equation}
where $\|\cdot\|$ is the operator norm from $h_{N}\times h_{N}\rightarrow h_{N}\times h_{N}.$

By Gronwall's inequality,
\begin{equation}\label{eq8.6}
\|z(t)-z_{0}\|_{N}\leq C(m+1)\varepsilon_{m}^{1-\frac{4(2n+3)}{N}}
\cdot\exp\left(\int_{0}^{t}\varepsilon_{m}\|B_{m}(\theta_{0}+\omega s)\|ds\right)\leq \varepsilon_{m}^{1/2}.
\end{equation}
Thus,
\begin{equation}\label{eq8.7}
\Psi_{m}: \mathbb{T}^{n}_{s_{m+1}}\times \Pi_{m+1}\rightarrow \mathbb{T}^{n}_{s_{m}}\times \Pi_{m},
\end{equation}
and
\begin{equation}\label{eq8.8}
\|\Psi_{m}-id\|_{h_{N}\to h_{N}}\leq \varepsilon_{m}^{1/2}.
\end{equation}
Since \eqref{eq8.1} is linear, so $\Psi_{m}$ is linear coordinate change. According to \eqref{eq8.2}, construct
Picard sequence:
$$\left\{
    \begin{array}{ll}
      z_{0}(t)=z_{0}, \\
      z_{j+1}(t)=z_{0}+\int_{0}^{t}\varepsilon_{m}B(\theta_{0}+\omega s)z_{j}(s)ds,\;j=0, 1, 2,\ldots.
    \end{array}
  \right.
$$
By \eqref{eq8.8}, this sequence with $t=1$ goes to
\begin{equation}\label{eq8.10}
\Psi_{m}(z_{0})=z(1)=(id+P_{m}(\theta_{0}))z_{0},
\end{equation}
where $id$ is the identity from $h_{N}\times h_{N}\rightarrow h_{N}\times h_{N},$ and $P(\theta_{0})$ is an operator form $h_{N}\times h_{N}\rightarrow h_{N}\times h_{N}$
for any fixed $\theta_{0}\in \mathbb{T}^{n}_{{s_{m+1}}}, \tau\in\Pi_{m+1},$
and is analytic in $\theta_{0}\in \mathbb{T}^{n}_{{s_{m+1}}},$ with
\begin{equation}\label{eq8.11}
\|P_{m}(\theta_{0})\|_{\mathbb{T}^{n}_{s_{m+1}}\times \Pi_{m+1}}\leq \varepsilon_{m}^{1/2}.
\end{equation}
Note that \eqref{eq8.1} is a Hamiltonian system. So $P_{m}(\theta_{0})$ is a symplectic linear operator from
$h_{N}\times h_{N}$ to $h_{N}\times h_{N}.$
%%%%%%%%%%%%%%%%%%%%%%%%%%%%%%%%%%%%%%%%%%%%%%%%%%%%%%%%%%%%%%%%%%%%%%%%%%%%%%%%%%%%%%%%%%%%
\subsection{Estimates of remainders}
The aim of this section is devoted to estimate the remainders:
$$R_{m+1}=\eqref{eq5.24}+\ldots + \eqref{eq5.25}.$$
\begin{itemize}
  \item Estimate of \eqref{eq5.24}.

By \eqref{eq40}, let
$$\widetilde{R}_{mm}=\widetilde{R}_{mm}(\theta)=\left(
            \begin{array}{cc}
           0 & \frac{1}{2}R_{m,m}(\theta) \\
              \frac{1}{2}R_{m,m}(\theta) & 0\\
            \end{array}
          \right),$$
then $$R_{mm}=\langle \widetilde{R}_{mm}\left(
                                          \begin{array}{cc}
                                            u \\
                                            \overline{u}
                                          \end{array}
                                        \right)
, \left(
 \begin{array}{cc}
  u \\
  \overline{u}
  \end{array}
  \right)\rangle.$$ So
$$(1-\Gamma_{K_{m}})R_{mm}\triangleq
\langle (1-\Gamma_{K_{m}})\widetilde{R}_{mm}
\left(
          \begin{array}{c}
   u \\
\overline{u} \\
     \end{array}
      \right),
\left(
          \begin{array}{c}
   u \\
\overline{u} \\
     \end{array}
      \right)
\rangle .$$
By the definition of truncation operator $\Gamma_{K_{m}},$
$$(1-\Gamma_{K_{m}})\widetilde{R}_{mm}=\sum_{|k|> K_{m}}\widehat{\widetilde{R}}_{mm}(k)e^{i \langle k, \theta\rangle},
\;\theta\in \mathbb{T}^{n}_{s_{m}},\;\tau\in \Pi_{m}.$$
Since $\widetilde{R}_{mm}=\widetilde{R}_{mm}(\theta)$ is analytic in $\theta\in \mathbb{T}^{n}_{s_{m}},$
\begin{eqnarray*}
&&\sup_{(\theta, \tau)\in\mathbb{T}^{n}_{s_{m+1}}\times \Pi_{m+1}}\|(1-\Gamma_{K_{m}})\widetilde{R}_{mm}\|_{l_{N}\to l_{N}}^{2}
\leq\sum_{|k|> K_{m}}\|\widehat{\widetilde{R}}_{mm}(k)\|_{N}^{2}e^{2|k|s_{m+1}}\\
&&\leq\|\widetilde{R}_{mm}\|_{\mathbb{T}_{s_{m}}^{n}\times \Pi_{m}}^{2}\sum_{|k|>K_{m}}e^{-2(s_{m}-s_{m+1})|k|}
%\leq 2\|J\widetilde{R}_{mm}J\|_{\mathbb{T}^{n}_{s_{m}}\times \Pi_{m}}^{2}e^{-2K_{m}(s_{m}-s_{m+1})}
\\
&&\leq \frac{C^{2}(m)e^{-2K_{m}(s_{m}-s_{m+1})}}{\varepsilon_{m}}\;\mbox{(by \eqref{eq30}}\\
&&\leq C^{2}(m)\varepsilon^{2}_{m}.
\end{eqnarray*}
That is,
$$\|(1-\Gamma_{K_{m}})\widetilde{R}_{mm}\|_{\mathbb{T}^{n}_{s_{m+1}}\times \Pi_{m+1}}\leq \varepsilon_{m}C(m).$$
Thus,
\begin{equation}\label{eq9.1}
\|\varepsilon_{m}(1-\Gamma_{K_{m}})\widetilde{R}_{mm}\|_{\mathbb{T}^{n}_{s_{m+1}}\times \Pi_{m+1}}
\leq \varepsilon^{2}_{m}C(m)\leq \varepsilon_{m+1}C(m+1).
\end{equation}
  \item Estimate of \eqref{eq5.28}.

Let
\begin{equation*}%\label{eq8.0}
S_{m}=\left(
            \begin{array}{cc}
             0 & \frac{1}{2}F(\theta, \tau) \\
            \frac{1}{2} F(\theta, \tau) & 0\\
            \end{array}
          \right),
\;\;\mathcal{J}=\left(\begin{array}{cc}
             0& -\sqrt{-1}id \\
            \sqrt{-1}id & 0\\
            \end{array}
          \right).
\end{equation*}
Then we can write
$$F=\langle S_{m}(\theta)\left(
                       \begin{array}{c}
                         u \\
                         \overline{u} \\
                       \end{array}
                     \right),
\left(
                       \begin{array}{c}
                         u \\
                         \overline{u} \\
                       \end{array}
                     \right)
\rangle =\langle S_{m}z, z\rangle,\;z=\left(
                       \begin{array}{c}
                         u \\
                         \overline{u} \\
                       \end{array}
                     \right).$$
Then
\begin{equation}\label{eq9.2}
%(6.16)=
\varepsilon_{m}^{2}\{R_{mm}, F\}
=4\varepsilon_{m}^{2} \langle {\widetilde{R}}_{mm}(\theta)JS_{m}(\theta)u, u\rangle.
\end{equation}
Noting $\mathbb{T}^{n}_{s_{m}}\times \Pi_{m}\supset \mathbb{T}^{n}_{s_{m+1}}\times \Pi_{m+1}.$ By \eqref{eq30} and \eqref{eq31} with $l=m, v=m,$
\begin{equation}\label{eq9.02}
\|\widetilde{R}_{mm}(\theta)\|_{\mathbb{T}^{n}_{s_{m+1}}\times \Pi_{m+1}}
\leq \|\widetilde{R}_{mm}(\theta)\|_{\mathbb{T}^{n}_{s_{m}}\times \Pi_{m}}\leq C(m),
\end{equation}
\begin{equation}\label{eq9.002}
\|\widetilde{R}_{mm}(\theta)\|^{\mathcal{L}}_{\mathbb{T}^{n}_{s_{m+1}}\times \Pi_{m+1}}\leq C(m).
\end{equation}
Let $\widetilde{S}_{m}(\theta)=\mathcal{J}S_{m}(\theta).$
Then by Lemmas \ref{lem7.1} in Section 7, we have
\begin{equation}\label{eq8.04}
\|\widetilde{S}_{m}(\theta)\|_{\mathbb{T}^{n}_{s_{m+1}}\times \Pi_{m+1}}\leq C(m+1)\varepsilon_{m}^{-\frac{2(2n+3)}{N}},
\end{equation}
\begin{equation}\label{eq8.05}
\| \widetilde{S}_{m}(\theta)\|^{\mathcal{L}}_{\mathbb{T}^{n}_{s_{m+1}}\times \Pi_{m+1}}\leq C(m+1)\varepsilon_{m}^{-\frac{4(2n+3)}{N}},
\end{equation}
and
\begin{equation}\label{eq9.3}
\|\widetilde{R}_{mm}JS_{m}\|_{\mathbb{T}^{n}_{s_{m+1}}\times \Pi_{m+1}}=\|\widetilde{R}_{mm}\widetilde{S}_{m}\|_{\mathbb{T}^{n}_{s_{m+1}}\times \Pi_{m+1}}\leq C(m)C(m+1)\varepsilon_{m}^{-\frac{2(2n+3)}{N}}.
\end{equation}
Note that the vector field is linear. So, by Taylor formula, one has
$$
\eqref{eq5.28}=\varepsilon_{m}^{2}\langle \widetilde{R}^{*}_{m}(\theta)u, u\rangle, $$
where
$$\widetilde{R}^{*}_{m}(\theta)=\sum_{j=1}^{\infty}\frac{4^{j}\varepsilon_{m}^{j-1}}{j!}\cdot\widetilde{R}_{mm}\underbrace{\widetilde{S}_{m}\cdots \widetilde{S}_{m}}_{j-fold}.$$
By \eqref{eq9.02} and \eqref{eq8.04},
\begin{eqnarray*}
\|\widetilde{R}^{*}_{m}(\theta)\|_{\mathbb{T}^{n}_{s_{m+1}}\times \Pi_{m+1}}
&\leq &\sum_{j=1}^{\infty}\frac{C(m)C(m+1)\varepsilon_{m}^{j-1}(\varepsilon_{m}^{-\frac{2(2n+3)}{N}})^{j}}{j!}\\
&\leq &C(m)C(m+1)\varepsilon_{m}^{-\frac{2(2n+3)}{N}}.
\end{eqnarray*}
By \eqref{eq9.002} and \eqref{eq8.05},
$$\|\widetilde{R}^{*}_{m}(\theta)\|^{\mathcal{L}}_{\mathbb{T}^{n}_{s_{m+1}}\times \Pi_{m+1}}\leq C(m)C(m+1)\varepsilon_{m}^{-\frac{4(2n+3)}{N}}.$$
Thus,
\begin{equation}\label{eq*1}
||\varepsilon_{m}^{2}\widetilde{R}_{m}^{*}||_{\mathbb{T}^{n}_{s_{m+1}}\times \Pi_{m+1}}
\leq C(m)C(m+1)\varepsilon_{m}^{2-\frac{2(2n+3)}{N}}\leq C(m+1)\varepsilon_{m+1},
\end{equation}
and
\begin{equation}\label{eq*2}
||\varepsilon_{m}^{2}\widetilde{R}_{m}^{*}||^{\mathcal{L}}_{\mathbb{T}^{n}_{s_{m+1}}\times \Pi_{m+1}}
\leq C(m)C(m+1)\varepsilon_{m}^{2-\frac{4(2n+3)}{N}}\leq C(m+1)\varepsilon_{m+1}.
\end{equation}

   \item Estimate of \eqref{eq5.27}

By \eqref{eq43},
\begin{equation}\label{eq9.7}
\{N_{m}, F\}=\langle [R_{mm}]u, \overline{u}\rangle-\Gamma_{K_{m}}R_{mm}-\omega\cdot \partial_{\theta}F\triangleq R_{mm}^{*}.
\end{equation}
Thus,
\begin{equation}\label{eq9.8}
\eqref{eq5.27}=\varepsilon_{m}^{2}\int_{0}^{1}(1-\tau)\{R_{mm}^{*}, F\}\circ X_{\varepsilon_{m}F}^{\tau} d\tau.
\end{equation}
Note $R_{mm}^{*}$ is a quadratic polynomial in $u$ and $\overline{u}.$ So we write
\begin{equation}\label{eq9.9}
R^{*}_{mm}=\langle \mathcal{R}_{m}(\theta, \tau)z, z\rangle,\; z=\left(
                                                                 \begin{array}{c}
                                                                   u \\
                                                                   \overline{u}\\
                                                                 \end{array}
                                                               \right).
\end{equation}
By \eqref{eq29} and \eqref{eq029} with $l=\nu=m,$ and using \eqref{eq8.04} and \eqref{eq8.05},
\begin{equation}\label{eq9.10}
\|\mathcal{R}_{m}\|_{\mathbb{T}^{n}_{s_{m}+1}\times \Pi_{m+1}}\leq C(m)\varepsilon_{m}^{-\frac{2(2n+3)}{N}},\;\;
\|\mathcal{R}_{m}\|^{\mathcal{L}}_{\mathbb{T}^{n}_{s_{m}+1}\times \Pi_{m+1}}\leq C(m)\varepsilon_{m}^{-\frac{4(2n+3)}{N}},
\end{equation}
where $\|\cdot\|$ is the operator norm in $h_{N}\times h_{N}\rightarrow h_{N}\times h_{N}.$
Recall $F=\langle S_{m}(\theta, \tau)u, u\rangle.$
Set
\begin{equation}\label{eq9.11}
[\mathcal{R}_{m}, \widetilde{S}_{m}]=2\mathcal{R}_{m}\widetilde{S}_{m}=2\mathcal{R}_{m}\mathcal{J}S_{m}.
\end{equation}
Using Taylor formula to \eqref{eq9.8}, we get
\begin{eqnarray*}
\eqref{eq5.27}&=&\frac{\varepsilon_{m}^{2}}{2!}\{\{R^{*}_{mm}, F\}, F\}+\ldots +\frac{\varepsilon_{m}^{j}}{j!}
\underbrace{\{\ldots\{R^{*}_{mm}, F\},\ldots, F\}}_{j-\mbox{fold}}+\ldots\\
&=&\Bigg\langle\left(\sum_{j=2}^{\infty}\frac{\varepsilon^{j}_{m}}{j!}\underbrace{[\ldots [\mathcal{R}_{m}, \widetilde{S}_{m}], \ldots , \widetilde{S}_{m}]}_{j-\mbox{fold}}\right)u, u\Bigg\rangle\\
&\triangleq &\langle \mathcal{R}^{**}(\theta, \tau)u,u\rangle.
\end{eqnarray*}
By \eqref{eq8.04},\eqref{eq9.10} and \eqref{eq9.11}, we have
\begin{eqnarray}\label{eq9.12}
\|\mathcal{R}^{**}(\theta, \tau)\|_{\mathbb{T}^{n}_{s_{m+1}}\times \Pi_{m+1}}
&\leq & \sum_{j=2}^{\infty}\frac{1}{j!}\|\mathcal{R}_{m}(\theta, \tau)\|_{\mathbb{T}^{n}_{s_{m}}\times \Pi_{m}}
(2\|\widetilde{S}_{m}\|_{\mathbb{T}^{n}_{s_{m+1}}\times \Pi_{m+1}}\varepsilon_{m})^{j}\nonumber\\
&\leq & \sum_{j=2}^{\infty}\frac{C(m)}{j!}\left(\varepsilon_{m}C(m+1)\varepsilon_{m}^{-\frac{2(2n+3)}{N}}\right)^{j}\nonumber\\
&\leq & C(m+1)\varepsilon_{m}^{4/3}=C(m+1)\varepsilon_{m+1}.
\end{eqnarray}
Similarly,
\begin{equation}\label{eq9.13}
\|\mathcal{R}^{**}\|^{\mathcal{L}}_{\mathbb{T}^{n}_{s_{m+1}}\times \Pi_{m+1}}\leq C(m+1)\varepsilon_{m+1}.
\end{equation}
  \item Estimate of \eqref{eq5.25}
%%%%%%%%%%%%%%%%%%%%%%%%%%%%%%%%%%%%%%%%%%%%%%%%%%%%%%%%%%%%%%%%%%%%%%%%%%%%%%%%%%%%%%%%%%%%%%%%%%
\begin{eqnarray}\label{eq9.19}
\eqref{eq5.25}=\sum_{l=m+1}^{\infty}\varepsilon_{l}(R_{lm}\circ X_{\varepsilon_{m}F}^{1}).
\end{eqnarray}
Write
$$R_{lm}=\langle \widetilde{R}_{lm}(\theta)u, u\rangle.$$
Then, by Taylor formula:
$$R_{lm}\circ X_{\varepsilon_{m}F}^{1}=R_{lm}+\sum_{j=1}^{\infty}\frac{1}{j!}\langle \widetilde{R}_{lmj} u, u\rangle,$$
where
$$\widetilde{R}_{lmj}=4^{j}\widetilde{R}_{lm}(\theta)\underbrace{\widetilde{S}_{m}(\theta)\cdots \widetilde{S}_{m}(\theta)}_{j-fold}\varepsilon_{m}^{j}.$$
By \eqref{eq30}, \eqref{eq31},
$$\|\widetilde{R}_{lm}\|_{\mathbb{T}^{n}_{s_{l}}\times \Pi_{m}}\leq C(l),\;\;\|\widetilde{R}_{lm}\|^{\mathcal{L}}_{\mathbb{T}^{n}_{s_{l}}\times \Pi_{m}}\leq C(l).$$
Combing the last inequalities with \eqref{eq8.04} and \eqref{eq8.05}, one has
\begin{eqnarray*}
&&\|\widetilde{R}_{lmj}\|_{\mathbb{T}^{n}_{s_{l}}\times \Pi_{m+1}}\\
&\leq & \|\widetilde{R}_{lm}\|_{\mathbb{T}^{n}_{s_{l}}\times \Pi_{m+1}}\cdot (||\widetilde{S}_{m}||_{\mathbb{T}^{n}_{m+1}\times \Pi_{m+1}}4\varepsilon_{m})^{j}\\
&\leq & C^{2}(m)(\varepsilon_{m}\varepsilon_{m}^{-\frac{2(2n+3)}{N}})^{j},
\end{eqnarray*}
and
\begin{eqnarray*}
&&||\widetilde{R}_{lmj}||^{\mathcal{L}}_{\mathbb{T}^{n}_{s_{l}}\times \Pi_{m+1}}\\
&\leq & ||\widetilde{R}_{lm}||^{\mathcal{L}}_{\mathbb{T}^{n}_{s_{l}}\times \Pi_{m+1}}(||\widetilde{S}_{m}||_{\mathbb{T}^{n}_{s_{l}}\times \Pi_{m+1}}4\varepsilon_{m})^{j}\\
&&+
||\widetilde{R}_{lm}||_{\mathbb{T}^{n}_{s_{l}}\times \Pi_{m+1}}(||\widetilde{S}_{m}||^{\mathcal{L}}_{\mathbb{T}^{n}_{s_{l}}\times \Pi_{m+1}}\varepsilon_{m})^{j}\\
&\leq &C^{2}(m)(\varepsilon_{m}\varepsilon_{m}^{-\frac{4(2n+3)}{N}})^{j}.
\end{eqnarray*}
Thus, let
$$\overline{R}_{l,m+1}:=R_{lm}+\sum_{j=1}^{\infty}\frac{1}{j!}\widetilde{R}_{lmj},$$
then
\begin{equation}\label{eq*3}
\eqref{eq5.25}=\sum_{l=m+1}^{\infty}\varepsilon_{l}\langle \overline{R}_{l,m+1}u, u\rangle
\end{equation}
and
\begin{equation}\label{eq*4}
||\overline{R}_{l,m+1}||_{\mathbb{T}^{n}_{s_{l}}\times \Pi_{m+1}}\leq C^{2}(m)\leq C(m+1),\;\;
||\overline{R}_{l,m+1}||^{\mathcal{L}}_{\mathbb{T}^{n}_{s_{l}}\times \Pi_{m+1}}\leq C^{2}(m)\leq C(m+1).
\end{equation}
As a whole, the remainder $R_{m+1}$ can be written  as
$$C_{m+1}R_{m+1}=\sum_{l=m+1}^{\infty}\varepsilon_{l}\langle R_{l, \nu}(\theta)u,\overline{u}\rangle,\;\;\nu=m+1,$$
where $R_{l\nu}$ satisfies \eqref{eq30} and \eqref{eq31} with
$\nu=m+1, l\geq m+1.$
This shows that the Assumption $(A2)_{\nu}$ with $\nu=m+1$ holds true.

By \eqref{eq5.23},
$$\mu_{j}^{(m)}=\widehat{R}_{mmjj}^{z\overline{z}}(0).$$
In \eqref{eq30} and \eqref{eq31}, we have
\begin{eqnarray*}
&&|\mu_{j}^{(m)}|_{\Pi_{m}}\leq
|R_{mmjj}(\theta, \tau)|\leq C(m),\\
&&|\mu_{j}^{(m)}|_{\Pi_{m}}^{\mathcal{L}}\leq |\partial_{\tau}R_{mmjj}(\theta, \tau)|\leq C(m).
\end{eqnarray*}
This shows that the Assumption $(A1)_{\nu}$ with $\nu=m+1$ holds true.
\end{itemize}
%%%%%%%%%%%%%%%%%%%%%%%%%%%%%%%%%%%%%%%%%%%%%%%%%%%%%%%%%%%%%%%%%%%%%%%%%%%%%%%%%%%%%%%%%%%%%%%%%
\subsection{Estimate of measure}\label{em}
In this section, $C$ denotes a universal constant, which may be different in different places.
Now let us return to \eqref{eq7.5}
\begin{equation}\label{eq10.01}
Q_{kij}^{(m)}\triangleq \left\{\tau\in\Pi_{m}\bigg| \big|-\langle k, \omega_{0}\rangle\tau
+\lambda_{i}^{(m)}-\lambda_{j}^{(m)}\big|<\frac{(|i-j|+1)\gamma_{m}}{A_{k}}\right\}.
\end{equation}
\begin{description}
  \item[Case 1.] If $i=j,$ one has $ k\neq 0.$%%%%%%%%%%%%%%%%%%%%%%%%%%%%%%%%%%%%%%%%%%%%%%%%%%%%

At this time,
\begin{equation}\label{eq10.02}
Q_{kii}^{(m)}= \left\{\tau\in\Pi_{m}\bigg| \big|\langle k, \omega_{0}\rangle\tau
\big|<\frac{\gamma_{m}}{A_{k}}\right\}.
\end{equation}
It follows
$$ \big|\langle k, \omega_{0}\rangle
\big|<\frac{\gamma_{m}}{|k|^{n+3}\tau}<\frac{\gamma}{2^{m}|k|^{n+3}}.$$
Recall $|\langle k, \omega_{0}\rangle |>\frac{\gamma}{|k|^{n+1}}$.
Then
\begin{equation}\label{eq10.03}
mes Q_{kii}^{(m)}=0.
\end{equation}
  \item[Case 2.] $i\neq j.$ %%%%%%%%%%%%%%%%%%%%%%%%%%%%%%%%%%%%%%%%%%%%%%%%%%%%%%%%%%%%%

If $ Q_{kij}^{(m)}=\varnothing,$ then $ mes Q_{kij}^{(m)}=0$. So we assume $ Q_{kij}^{(m)}\neq\varnothing,$ in the following. Then $\exists \,\tau\in \Pi_{m}$ such that
\begin{equation}\label{eq10.04}
|-\langle k, \omega_{0}\rangle \tau+\lambda_{i}^{(m)}-\lambda_{j}^{(m)}|<\frac{|i-j|+1}{A_{k}}\gamma_{m}.
\end{equation}
It follows from \eqref{eq28} and \eqref{eq29} that
\begin{equation}\label{eq10.05}
\lambda_{i}^{(m)}-\lambda_{j}^{(m)}=i^{2}-j^{2}+O (\varepsilon_{0})\geq \frac{2}{3}|i^{2}-j^{2}|.
\end{equation}
When $|i|\geq C |k|\gg|\langle k, \omega\rangle|$ or $|j|\geq C |k|\gg|\langle k, \omega\rangle|$.%%%%%%%%%%%%%%
By \eqref{eq10.05}, one has $$\big|-\langle k, \omega\rangle
+\lambda_{i}^{(m)}-\lambda_{j}^{(m)}\big|\geq \frac{2}{3}|i+j||i-j|-|\langle k, \omega\rangle|\geq
\frac{1}{2}|i+j||i-j|,$$ which implies $ Q_{kij}^{(m)}=\varnothing,$ then
\begin{equation}\label{eq10.013}
mes Q_{kij}^{(m)}=0.
\end{equation}
Now assume $$|i|< C |k| \;\;\mbox{and} \;\;|j|< C |k|.$$%%%%%%%%%%%%%%%%%%%%%%%%%%%%%%%%%%%%%%%%%%%%%%%%%%%%%%%
Note that
$$-\langle k, \omega\rangle +i^{2}-j^{2}=-\langle k, \omega_{0}\rangle\tau +i^{2}-j^{2}=\tau(-\langle k, \omega_{0}\rangle+\frac{i^{2}-j^{2}}{\tau})$$
and
\begin{equation}\label{eq*}
\big|\frac{d}{d \tau}(-\langle k, \omega_{0}\rangle+\frac{i^{2}-j^{2}}{\tau})\big|=\frac{|i^{2}-j^{2}|}{\tau^{2}}\geq \frac{1}{4}|i^{2}-j^{2}|.
\end{equation}
It follows that
\begin{equation}\label{eq10.011}
mes Q_{kij}^{(m)}\leq \frac{8}{|i^{2}-j^{2}|}\left(\frac{|i^{2}-j^{2}|+1}{A_{k}}\gamma_{m}+C_{1}\varepsilon_{0}
\right).
\end{equation}
Then
\begin{eqnarray}\label{eq10.015}
&&mes \bigcup_{|k|\leq K_{m}}\bigcup_{\begin{array}{c}
                                i\leq C|k| \nonumber\\
                                j\leq C|k|
                              \end{array}}Q^{(m)}_{kij}\nonumber\\
&\leq & \sum_{|k|\leq K_{m}}\frac{C\gamma_{m}}{A_{k}}\sum_{i,j=1}^{C|k|}\frac{1}{i+j}
\leq  \sum_{|k|\leq K_{m}}\frac{C|k|^{2}\gamma_{m}}{A_{k}}
\leq  C\gamma_{m}.
\end{eqnarray}

Combining \eqref{eq10.03}, \eqref{eq10.013} and \eqref{eq10.015}, we have
\begin{eqnarray}\label{eq10.016}
mes \bigcup_{|k|\leq K_{m}}\bigcup_{i=1}^{\infty}\bigcup_{j=1}^{\infty}Q^{(m)}_{kij}
\leq C\gamma_{m}.
 \end{eqnarray}
 Let
 $$\Pi_{m+1}=\Pi_{m}\backslash
 \bigcup_{|k|\leq K_{m}}\bigcup_{i,j=1}^{\infty}Q^{(m)}_{kij}.$$
 Then we have proved the following Lemma \ref{lem10.1}.
\end{description}
\begin{lemma}\label{lem10.1}
$$mes\Pi_{m+1}\geq mes \Pi_{m}-C\gamma_{m}.$$
\end{lemma}
%%%%%%%%%%%%%%%%%%%%%%%%%%%%%%%%%%%%%%%%%%%%%%%%%%%%%%%%%%%%%%%%%%%%%%%%%%%%%%%%%%%%%%%%%%%%%%%%%%%%%
\section{Proof of Theorems}
Let
$$\Pi_{\infty}=\bigcap_{m=1}^{\infty} \Pi_{m},$$
and $$\Psi_{\infty}=\lim_{m\rightarrow \infty}\Psi_{0}\circ\Psi_{1}\circ \cdots \circ \Psi_{m}.$$
By \eqref{eq33} and \eqref{eq33H}, one has
$$\Psi_{\infty}: \mathbb{T}^{n}\times \Pi_{\infty}\rightarrow \mathbb{T}^{n}\times \Pi_{\infty},$$
$$||\Psi_{\infty}-id||\leq \varepsilon_{0}^{1/2},$$
and, by \eqref{eq34},
$$H_{\infty}=H \circ \Psi_{\infty}=\sum_{j=1}^{\infty}\lambda_{j}^{\infty}Z_{j}\overline{Z}_{j},$$
where
$$\lambda_{j}^{\infty}=\lim_{m\rightarrow \infty}\lambda_{j}^{(m)}.$$
By \eqref{eq28} and \eqref{eq29}, the limit $\lambda_{j}^{\infty}$ does exists and
$$\lambda_{j}^{\infty}=j^{2}+M+O (\varepsilon_{0}).$$
This completes the proof of Theorem \ref{thm1.1}.
%%%%%%%%%%%%%%%%section  Acknowledgements%%%%%%%%%%%%%%%%%%%%%%%%%%%%%%%%%%%%%%%%%%%%%%%%%%%%%%%%%
\section*{Acknowledgements}
The work was supported in part by National
Nature Science Foundation of China (11601277).
%%%%%%%%%%%%%%%%%%%%%%%%%%%%%%%%%%%%%%%%%%%%%%%%%%%%%%%%%%
%%%%%%%%%%%%%%%%%%%%%%%%%%%%%%%%%%%%%%%%%%%%%%%%%%%%%%%%%%


\begin{thebibliography}{00}

\bibitem{Bambusi17arxiv}D. Bambusi, B. Gr\'{e}bert, A. Maspero, D. Robert, Reducibility of the quantum harmonic oscillator in d-dimensions with polynomial time dependent perturbation, arXiv:1702.05274v1 [math. AP], 2017.

\bibitem{Bambusi17CMP}D. Bambusi, Reducibility of 1-d Schr\"{o}dinger equation with time quasiperiodic unbounded perturbations, II, Commun. Math. Phys. 353, 353-378, 2017.

\bibitem{Bogoljubov1969}N. N. Bogoljubov, Ju. A. Mitropoliskii, and A. M. Samo$\breve{{\i}}$lenko, Methods of accelerated convergence in nonlinear mechanics, Springer-Verlag, New York, 1976, translated from the original Russian, Naukova Dumka, Kiev, 1969.

\bibitem{Bourgain-Goldstin2000}J. Bourgain, M. Goldstein, On nonperturbative localization
with quasi-periodic potential, Ann. of Math. 152 (2000) 835-879.

\bibitem{Chierchia} L. Chierchia, D.Qian, Moser¡¯s theorem for lower dimensional tori. J. Diff. Eqs. 206, 55-93, 2004.

\bibitem{Combescure1987Ann}M. Combescure, The quantum stability problem for time-periodic perturbations of the harmonic oscillator, Ann. Inst. Henri Poincar\'{e}, 47 (1), 63-83, 1987.

 \bibitem{Duclos-Stovicek96}P. Duclos and P. Stovicek. Floquet Hamiltonians with pure point spectrum. Comm.
Math. Phys., 177(2), 327¨C347, 1996.

\bibitem{Duclos-Stovicek02}P. Duclos, O. Lev, P. Stovicek, and M. Vittot. Perturbation of an eigen-value from a dense point spectrum: A general Floquet Hamiltonian, Ann. Inst. Henri Poincar\'e, 71(3), 241-301, 1999




\bibitem{eliasson-kuksin} H. L. Eliasson, S. B. Kuksin , On reducibility of Schr\"odinger equations with quasiperiodic in time potentials, Communications in Mathematical Physics, 2009.

\bibitem{Klein2005}S. Klein, Anderson localization for the discrete one-dimensional quasi-periodic Schr\"{o}dinger operator with potential defined by a Gevrey-class function, J. Funct. Anal. 218(2) (2005) 255-292.

\bibitem{Kuk1} S. B. Kuksin, Nearly integrable infinite-dimensional Hamiltonian systems.( Lecture Notes in Math. 1556). Springer-Verlag, New York, 1993.

%\bibitem{Liang17JDDE}Z. Liang, X. Wang, On reducibility of 1d wave equation with quasiperiodic in time potentials, J Dyn. Diff. Equat., DOI:10.1007/s10884-017-9576-4.

%\bibitem {Liang-arxiv}Z. Liang, Z. Wang, Reducibility of quantum harmonic oscillator on $\mathbb{R}^{d}$ with differential and quasi-periodic in time potential, arXiv:1704.06744v1 [math. DS], 2017.

\bibitem{Poschel1996}J. P\"oschel, A KAM-theorem for some nonlinear PDEs,
 Ann. Scuola Norm. Sup. Pisa, Cl. Sci., $\quad$IV Ser. 15,
     23, 119-148, 1996.
     
%\bibitem{Riccardo Montalto} R. Montalto, A reducibility result for a class of linear wave
equations with unbounded perturbations on,arXiv:1702.06880v1[math. AP], 2017.

\bibitem{Salamon1989}D. Salamon, E. Zehnder, KAM theory in configuration space, Commun. Math. Helv. 64, 84-132, 1989.

\bibitem{Salamon2004}D. Salamon, The Kolmogorov-Arnold-Moser theorem, Mathematical Physics Electronic Journal, 10(3),2004.



%\bibitem{E. Zehnder} E. Zehnder,
Generalized implicit function theorems with applications to small divisor problems I II,
Comm. Pure Appl. Math., 28 (1975), 91-140; 29 (1976) 49¨C113.





\bibitem{Wang17nonlinearity}Z. Wang, Z. Liang, Reducibility of 1D quantum harmonic oscillator perturbed by a quasiperiodic potential with logarithmic decay, Nonlinearity, 30, 1405-1448, 2017.

\bibitem{yuan-zhang}X. Yuan, K. Zhang, A reduction theorem for time dependent Schr¡§odinger operator with finite differentiable unbounded perturbation. J. Math. Phys. 54(5), 465-480, 2013.




\end{thebibliography}
\end{document}